\documentclass[11pt,a4paper]{article}

\usepackage{epsf,epsfig,amsfonts,amsgen,amsmath,amstext,amsbsy,amsopn,amsthm,cases,listings,color
}
\usepackage{ebezier,eepic}
\usepackage{color}
\usepackage{multirow}
\usepackage{epstopdf}
\usepackage{graphicx}
\usepackage{pgf,tikz}
\usepackage{mathrsfs}
\usepackage[marginal]{footmisc}
\usepackage{enumitem}
\usepackage[titletoc]{appendix}
\usepackage{booktabs}
\usepackage{url}
\usepackage{mathtools}

\usepackage{pgfplots}
\usepackage{authblk}
\usepackage{amssymb}

\usepackage{wasysym}

\usepackage{empheq}

\usepackage{dsfont}
\usepackage{subfigure}
\pgfplotsset{compat=1.18}
\usepackage{mathrsfs}
\usepackage{wasysym} 
\usetikzlibrary{arrows}

\usepackage{algorithm}
\usepackage{algpseudocode}

\usepackage{mathtools}

\DeclarePairedDelimiter{\floor}{\lfloor}{\rfloor}
\usepackage{algcompatible}
\usepackage[backref=page]{hyperref}
\usepackage{bm}

\allowdisplaybreaks[1]

\definecolor{uuuuuu}{rgb}{0.27,0.27,0.27}
\definecolor{sqsqsq}{rgb}{0.1255,0.1255,0.1255}

\newtheorem{definition}{Definition} [section]
\newtheorem{theorem}[definition]{Theorem}
\newtheorem{lemma}[definition]{Lemma}
\newtheorem{proposition}[definition]{Proposition}

\newtheorem{conjecture}[definition]{Conjecture}
\newtheorem{claim}[definition]{Claim}

\newtheorem{fact}[definition]{Fact}
\newenvironment{pf}{\noindent{\bf Proof}}

\newcommand{\hide}[1]{}

\begin{document}
\title{\bf\Large A note on extremal constructions\\ for the Erd{\H o}s--Rademacher problem}

\date{\today}
\author{Xizhi Liu\thanks{Research was supported by ERC Advanced Grant 101020255 and
            Leverhulme Research Project Grant RPG-2018-424. Email: \texttt{xizhi.liu.ac@gmail.com}} }
\author{Oleg Pikhurko\thanks{Research was supported by ERC Advanced Grant 101020255 and
            Leverhulme Research Project Grant RPG-2018-424. Email: \texttt{o.pikhurko@warwick.ac.uk}} }
\affil{Mathematics Institute and DIMAP,\\
            University of Warwick,\\
            Coventry, CV4 7AL, UK}
            
\maketitle
\begin{abstract}
For given positive integers $r\ge 3$, $n$ and $e\le \binom{n}{2}$, the famous Erd\H os--Rademacher problem asks for the minimum number of $r$-cliques in a graph with $n$ vertices and $e$ edges. A conjecture of Lov\'asz and Simonovits from the 1970s states that, for every $r\ge 3$, if $n$ is sufficiently large then, for every $e\le \binom{n}{2}$, at least one extremal graph can be obtained from a complete partite graph by adding a triangle-free graph into one part. 
 
In this note, we explicitly write the minimum number of $r$-cliques predicted by the above conjecture. Also, we describe what we believe to be the set of extremal graphs for any $r\ge 4$ and all large~$n$, amending the previous conjecture of Pikhurko and Raz\-bo\-rov.
\medskip

\noindent\textbf{Keywords:} Erd{\H o}s--Rademacher problem,  Lov{\' a}sz--Simonovits conjecture, Clique density theorem. 
\end{abstract}

\section{Introduction}\label{SEC:introduction}
Given integers $n \ge r \ge 2$, 
let $T_{r}(n)$ denote the balanced complete $r$-partite graph on $n$ vertices, and let $t_{r}(n)$ denote the number of edges in $T_{r}(n)$. 
The celebrated Tur\'{a}n Theorem~\cite{T41} (with the case $r=3$ proved earlier by Mantel~\cite{Mantel07}) states that, for $n \ge r \ge 3$, every $n$-vertex graph with at least $t_{r-1}(n)+1$ edges contains a copy of an \emph{$r$-clique} $K_{r}$, that is, a complete graph on $r$ vertices.  
An unpublished result of Rademacher from 1941 (see~\cite{Erdos55}) states that, in fact, every $n$-vertex graph with $t_{2}(n)+1$ edges contains at least $\floor*{n/2}$ copies of $K_{3}$. The graph obtained from $T_{2}(n)$ by adding one edge to the larger part shows that the bound $\floor*{n/2}$ is tight.   
Rademacher's theorem motivated Erd{\H o}s~\cite{Erdos55} to consider the following more general question, now referred to as the \emph{Erd{\H o}s--Rademacher problem}: determine 
\begin{align}\label{equ:Erdos-Rademacher-problem}
    g_{r}(n,e) := \min\Big\{N(K_r,G) \colon \text{$G$ is an $(n,e)$-graph}\Big\},
\end{align}
 where an \emph{$(n,e)$-graph} means a graph with $n$ vertices and $e$ edges and
$N(K_r,G)$ denotes the number of $r$-cliques in~$G$. 

This problem
has attracted a lot of attention and has been actively studied since it first appeared. 
Various results covering special ranges of $(n,e)$  were obtained (see e.g.~\cite{Goo59,Erd62,MM62,NS63,Bol76,LS76,Nik76,NK81,LS83,Fis89,GS00}) until Razborov~\cite{Raz08} determined the asymptotic value of $g_{3}(n,e)$ using flag algebras. 
Later, using different methods, Nikiforov~\cite{Nik11} determined the asymptotic value of $g_r(n,e)$ for $r=4$ and  Reiher~\cite{Rei16} did this for all $r \ge 5$. For some further related results, we refer the reader to~\cite{Mub10,Mub13,PR17,KLPS20,XK21,LM22,BC23}.

Determining the exact value of $g_{r}(n,e)$ seems very challenging  due to  multiple (conjectured) extremal constructions. 
Given $n$ and $e$ in $\mathbb{N}:=\{1,2,\dots\}$ with $e \le \binom{n}{2}$, let 
    \begin{align}\label{equ:def-k}
        k = k(n,e) := \min\left\{s\in \mathbb{N} \colon t_{s}(n) \ge e\right\},
    \end{align}
    that is, $k$ is the smallest chromatic number that an $(n,e)$-graph can have.
Let $\mathcal{H}_{1}(n,e)$ (resp.\ $\mathcal{K}(n,e)$) denote the family of $(n,e)$-graphs that can be obtained from a complete $(k-1)$-partite (resp.\ complete multipartite) graph by adding a triangle-free graph into one part. 
Note that the only difference between these two definitions is that we restrict the number of parts to $k-1$ when defining $\mathcal{H}_{1}(n,e)$; thus $\mathcal{H}_1(n,e) \subseteq \mathcal{K}(n,e)$.    
%
Lov\'{a}sz and Simonovits~\cite{LS76} conjectured that for every integer $r\ge 3$ there exists $n_0$ such that, for all positive integers $n \ge n_0$ and $e\le \binom{n}{2}$, it holds that 
\begin{align}\label{equ:LS-conjecture}
    g_{r}(n,e)
    = \min\Big\{N(K_r, H) \colon H \in \mathcal{K}(n,e)\Big\},
\end{align}
 that is, at least one $g_r(n,e)$-extremal graph is in~$\mathcal{K}(n,e)$. Note that~\eqref{equ:LS-conjecture} trivially holds for $e\le t_{r-1}(n)$ when $g_r(n,e)=0$.

Erd{\H o}s in~\cite{Erdos55} (resp.\ \cite{Erd62}) showed that~\eqref{equ:LS-conjecture} is true for $r = 3$ when $e \le t_2(n) + 3$ (resp.\ $e\le t_2(n)+cn$ for some constant $c>0$). 
Lov\'{a}sz and Simonovits~\cite{LS76} (see also Nikiforov and Khadzhiivanov~\cite{NK81}) extended the result of Erd{\H o}s to all $e$ satisfying $e \le t_2(n) + \lfloor n/2 \rfloor$. 
Later, Lov\'{a}sz and Simonovits~\cite{LS83} proved~\eqref{equ:LS-conjecture} for $r \ge 3$ when $e/\binom{n}{2}$ lies in a small upper neighborhood of $1-1/m$ for some integer $m \ge r-1$. 
More recently, Liu, Pikhurko and Staden~\cite{LPS20} determined $g_{3}(n,e)$ for all positive integers $n$ when $e \le (1-o(1))\binom{n}{2}$. 
Determining the exact value of $g_{r}(n,e)$ for $r \ge 4$ is still wide open in general.  

%
Given  $n, e\in \mathbb N$ with $e \le \binom{n}{2}$, 
let $\bm{a}^{\ast}=\bm{a}^{\ast}(n,e) \in \mathbb{N}^{k}$ be the unique vector such that 
    \begin{align*}
         a_{k}^{\ast} := \min\left\{a\in \mathbb{N} \colon a(n-a) + t_{k-1}(n-a) \ge e\right\},  \\
         a_{1}^{\ast}+\dots+a_{k-1}^{\ast}  = n-a_{k}^{\ast},
                \quad\text{and}\quad 
        a_{1}^{\ast} \ge \dots \ge a_{k-1}^{\ast} \ge a_{1}^{\ast}-1, 
    \end{align*}
where $k = k(n,e)$ is as defined in~\eqref{equ:def-k}. Thus $a_k^{\ast}$ is the smallest possible part size that a $k$-partite $(n,e)$-graph can have.
Also, let 
        \begin{align*}
            m^{\ast}  = m^{\ast}(n,e) & := \sum_{\{i,j\}\in \binom{[k]}{2}}a_{i}^{\ast}a_{j}^{\ast} -e, 
            \quad\text{and}\quad \\
            h^{\ast}_{r}(n,e) & := \sum_{I\in \binom{[k]}{r}}\prod_{i\in I}a_i^{\ast} - m^{\ast}\cdot \sum_{I'\in \binom{[k-2]}{r-2}}\prod_{j\in I'}a_j^{\ast}, 
        \end{align*}
 where $[k]:=\{1,\dots,k\}$ and ${X\choose k}:=\{Y\subseteq X: |Y|=k\}$.
Let $T:= K[A_{1}^{\ast}, \ldots, A_{k}^{\ast}]$ be the complete $k$-partite graph with parts $A_1^{\ast},\dots,A_{k}^{\ast}$ where $|A_{i}^{\ast}| = a_{i}^{\ast}$ for $i\in [k]$.
Let $H^{\ast} = H^{\ast}(n,e)$ be the graph obtained from $T$  by removing an $m^{\ast}$-edge star whose centre lies in $A_{k}^{\ast}$ and whose leaves lie in $A_{k-1}^{\ast}$.
It is not hard to see (see e.g.\ the calculation in~\eqref{eq:m'Ineq}) that $0\le m^{\ast}\le a_{k-1}^{\ast}-a_{k}^{\ast}$, so the graph $H^{\ast}$ is well-defined.
Also, let $\mathcal{H}_{1}^{\ast}(n,e)$ be the family defined as follows:
If $m^{\ast} = 0$,  take all graphs obtained from $T$ by replacing, for some $i\in [k-1]$, the bipartite graph $T[A_{i}^{\ast}\cup A_{k}^{\ast}]$ with an arbitrary triangle-free graph with $a_{i}^{\ast}a_{k}^{\ast}$ edges. 
If $m^{\ast}>0$,  take all graphs obtained from $T$ by replacing $T[A_{k-1}^{\ast}\cup A_{k}^{\ast}]$ with an arbitrary triangle-free graph with $a_{k-1}^{\ast}a_{k}^{\ast} - m^{\ast}$ edges. 
Observe that $\mathcal{H}_1^{\ast}(n,e)
\subseteq \mathcal{H}_1(n,e)$ and every graph in $\mathcal{H}_{1}^{\ast}(n,e)$ has the same number of $r$-cliques (see Fact~\ref{FACT:formula-count-Kr-1-part}); also, the graph $H^*=H^*(n,e)$ is contained in~$\mathcal{H}_1^{\ast}(n,e)$.

Sharpening the Lov\'{a}sz--Simonovits Conjecture,  Pikhurko and Razborov~\cite[Conjecture~1.4]{PR17} conjectured that, for $r\ge 4$ and sufficiently large $n$, every $n$-vertex graph with $e\le \binom{n}{2}$ edges and that contains the minimum number of $K_r$ is  in  $\mathcal{K}(n,e)$.
However, we show here that this conjecture is false (see Theorem~\ref{THM:N(Kr,G)-H1} and Proposition~\ref{PROP:N(K_r,H)-H2-ast}) and present an amended version (see Conjecture~\ref{CONJ:amended-PR}) as follows.

First, we write explicitly the value of $g_{r}(n,e)$ predicted by the Lov\'{a}sz--Simonovits Conjecture. 
(We also refer the reader to \cite[Proposition~1.5]{LPS20} where similar results are proved for $r=3$.)

\begin{theorem}\label{THM:N(Kr,G)-H1}
    Suppose that $r,n,e\in\mathbb N$ satisfy $n \ge r \ge 3$ and $e \le \binom{n}{2}$. 
    Then 
    \begin{align}\label{eq:N(Kr,G)-H1}
        \min\Big\{N(K_r, G) \colon G\in \mathcal{K}(n,e)\Big\} = h_{r}^{\ast}(n,e).
    \end{align}
    Moreover, if $r\ge 4$ and $e> t_{r-1}(n)$, then 
    \begin{align}
        \Big\{G \in \mathcal{K}(n,e) \colon N(K_r, G) = h_{r}^{\ast}(n,e)\Big\}
        & =  \mathcal{H}_{1}^{\ast}(n,e). \label{eq:MinFamily}
     \end{align}
\end{theorem}

Note that, since $\mathcal{H}_1^{\ast}(n,e)\subseteq \mathcal{H}_1(n,e)$, Theorem~\ref{THM:N(Kr,G)-H1} remains true if we replace $\mathcal{K}(n,e)$ by $\mathcal{H}_1(n,e)$. In fact, the later version of the Lov\'asz--Simonovits Conjecture from~\cite{LS83} states that, for all sufficiently large $n\ge n_0(r)$, at least one $g_r(n,e)$-extremal graph is in $\mathcal{H}_1(n,e)$. By~\eqref{eq:N(Kr,G)-H1},
these two conjectures are equivalent. 
One should be able to show with some extra work that~\eqref{eq:MinFamily} also holds for $r=3$ (it is also implied by the results in~\cite{LPS20} that~\eqref{eq:MinFamily} holds for most $e$, given $n$). Since our main focus is the case $r\ge 4$, we do not pursue this strengthening here.
 
Given integers $n,e \in \mathbb{N}$ with $e\le \binom{n}{2}$, we define the family $\mathcal{H}_{2}^{\ast}(n,e)$ as follows (with $k, \bm{a}^{\ast}, m^{\ast}$ being as before).
Take those graphs in $\mathcal{H}_{1}^{\ast}(n,e)$ that are $k$-partite, along with the following family. 
Take disjoint sets $A_{1}, \ldots, A_{k}$ of sizes $a_{1}^{\ast}, \ldots, a_{k}^{\ast}$, respectively, and let $m:= m^{\ast}$. 
If $m^{\ast} = 0$ and $a_{1}^{\ast} \ge a_{k}^{\ast}+2$, then we also allow $\left(|A_{1}|, \ldots, |A_{k}|\right) = \left(a_{2}^{\ast}, \ldots, a_{k-1}^{\ast}, a_{1}^{\ast}-1, a_{k}^{\ast}+1\right)$ and let $m:=a_1^{\ast}-a_k^{\ast}-1$. 
Take all graphs obtained from $K[A_{1}, \ldots, A_{k}]$ by removing any $m$ edges, each connecting $B_i$ to $A_i$ for some $i\in I$, where $I:= \left\{i\in [k-1] \colon |A_{i}| = |A_{k-1}|\right\}$ and $\left\{B_i\colon i\in I\right\}$ are some pairwise disjoint subsets of $A_{k}$. Clearly,  every graph in $\mathcal{H}_{2}^{\ast}(n,e)$ is an $(n,e)$-graph.

\begin{proposition}\label{PROP:N(K_r,H)-H2-ast}
    Suppose that $n \ge r \ge 4$ and $t_{r-1}(n)< e \le \binom{n}{2}$ are integers. 
    Then 
    \begin{align*}
        N(K_r, G)  
        = h_{r}^{\ast}(n,e), \quad\text{for every\ } G\in \mathcal{H}_{2}^{\ast}(n,e). 
    \end{align*}
    Also, there are infinitely many pairs $(n,e) \in \mathbb{N}^2$ with $t_{r-1}(n) < e \le \binom{n}{2}$ such that $\mathcal{H}_{2}^{\ast}(n,e) \setminus \mathcal{H}_{1}^{\ast}(n,e) \neq \emptyset$. 
\end{proposition}

We propose the following amended conjecture.

\begin{conjecture}\label{CONJ:amended-PR}
    Let $r \ge 4$ be fixed. For every sufficiently large integer $n$ and every integer $e$ with $t_{r-1}(n)< e \le \binom{n}{2}$, it holds that 
    \begin{align*}
        \Big\{G\colon \text{$G$ is an $(n,e)$-graph  with $N(K_r, G) = g_{r}(n,e)$} \Big\}
        = \mathcal{H}_{1}^{\ast}(n,e) \cup \mathcal{H}_{2}^{\ast}(n,e). 
    \end{align*}
\end{conjecture}

For comparison with the case $r=3$, the exact result of Liu, Pikhurko and Staden~\cite{LPS20} valid for $e\le (1-o(1))\binom{n}{2}$ states that the set of $g_3(n,e)$-extremal graphs is exactly  $\mathcal{H}_{0}^{\ast}(n,e)\cup \mathcal{H}_{2}^{\ast}(n,e)$ for a certain explicit family $\mathcal{H}_{0}^{\ast}(n,e)\supseteq \mathcal{H}_{1}^{\ast}(n,e)$, where the inclusion is strict for infinitely many pairs $(n,e)$.  However, for $r\ge 4$ and $e>t_{r-1}(n)$, every graph in $\mathcal{H}_{0}^{\ast}(n,e)\setminus \mathcal{H}_{1}^{\ast}(n,e)$ can be shown to have more $K_r$'s than $H^{\ast}(n,e)$. (Basically, each such graph is obtained from a complete $(k-1)$-partite graph by adding edges into more than one part and cannot minimise the number of $K_r$'s for $r\ge 4$ by Lemma~\ref{LEMMA:one-partially-full-part}.)  

For the purposes of this paper (namely for
Proposition~\ref{PROP:N(K_r,H)-H2-ast}), only the difference $\mathcal{H}_{2}^{\ast}(n,e)\setminus \mathcal{H}_{1}^{\ast}(n,e)$ matters; we use the current definitions merely so that the families $\mathcal{H}_i^{\ast}(n,e)$ and $\mathcal{H}_i(n,e)$ are the same as in~\cite{LPS20}.

The rest of the paper of organised as follows. In the next section, we present some definitions and preliminary results. As a step towards proving Theorem~\ref{THM:N(Kr,G)-H1}, we first find extremal graphs in a certain family $\mathcal{H}_0(n,e)$ in Section~\ref{se:H0} (see Proposition~\ref{PROP:H0-min-H1-ast} for the exact statement). We derive  Theorem~\ref{THM:N(Kr,G)-H1} in Section~\ref{SUBSEC:proof-Thm}. The proof of Proposition~\ref{PROP:N(K_r,H)-H2-ast} is presented in Section~\ref{SUBSEC:proof-Prop}.

\section{Preliminaries}\label{SEC:prelim}
Given $\ell$ pairwise disjoint sets $A_1, \ldots, A_{\ell}$, we use $K[A_1, \ldots, A_{\ell}]$ to denote the complete ${\ell}$-partite graph with parts $A_1, \ldots, A_{\ell}$; if we care only about the isomorphism type of this graph (i.e.\ only the sizes of the parts matter), we may instead write $K_{a_1, \ldots, a_{\ell}}$, where $a_i:=|A_i|$ for $i\in [\ell]$.

Let $G=(V,E)$ be a graph. By
 $|G|$ we denote the number of edges in~$G$.  Let $\overline{G} := \left(V,\binom{V}{2}\setminus E\right)$ denote the \emph{complement} of $G$. 
The subgraph of $G$ \emph{induced} by a set $A\subseteq V$ is  $G[A]:=\left(A,\binom{A}{2}\cap E\right)$. For disjoint $A,B\subseteq V$, we use $G[A,B]$ to denote the induced bipartite graph with parts $A$ and $B$ (which consists of edges connecting $A$ to $B$). 

In the remainder of this note, we assume  unless it is stated otherwise that $r,n,e\in\mathbb N$ satisfy $r\ge 3$ and $e\le \binom{n}{2}$ (and we minimise the number of $r$-cliques over $(n,e)$-graphs). Also,   $k=k(n,e)$ is defined in~\eqref{equ:def-k}. 

Given a family $\mathcal{F}$ of $(n,e)$-graphs, we use $\mathcal{F}^{\min}$ to denote the collection of graphs $F \in \mathcal{F}$ with the minimum number of $K_r$'s (over all graphs in $\mathcal{F}$).  
For convenience, we set $N(K_0, G) := 1$ and $N(K_{-1}, G) := 0$ for all graphs $G$. 

Let the family $\mathcal{H}_{0}(n,e)$ be the collection of all $(n,e)$-graphs that can be obtained from an $n$-vertex complete $(k-1)$-partite graph by adding a (possibly empty) triangle-free graph into each part. 
It is clear from the definition that $\mathcal{H}_1(n,e) \subseteq \mathcal{H}_{0}(n,e)$.

The following fact follows from some simple calculations (with the argument for Part~(i) being the same as in~\eqref{eq:m'Ineq}). 
\begin{fact}\label{FACT:k,a,m,H,h}
    Let $k, \bm{a}^{\ast}, m^{\ast}, H^{\ast}$, and $h_r^{\ast}(n,e)$ be as defined in Section~\ref{SEC:introduction}. 
    Then it holds for all $r \ge 3$ that
    \begin{enumerate}[label=(\roman*)]
        \item $0 \le m^{\ast} \le a_{k-1}^{\ast} - a_{k}^{\ast}$, 
        \item $|K_{a_{1}^{\ast}, \ldots, a_{k}^{\ast}}| - |K_{a_{1}^{\ast}, \ldots, a_{k-2}^{\ast}, a_{k-1}^{\ast}+1, a_{k}^{\ast}-1}|  = a_{k-1}^{\ast} - a_{k}^{\ast}+1$, 
        \item $N(K_r, H^{\ast}) = h^{\ast}_{r}(n,e) \ge g_r(n,e)$. 
    \end{enumerate}
\end{fact}
We also need the following simple facts for counting $r$-cliques in some special classes of graphs. 
\begin{fact}\label{FACT:formula-count-Kr-1-part}
    Let $G$ be a graph, $S \subseteq V(G)$ be a vertex set, and $\overline{S}:= V(G)\setminus S$. 
    Suppose that the induced subgraph $G[S]$ is triangle-free, and the induced bipartite graph $G[S,\overline{S}]$ is complete. 
    Then 
    \begin{align*}
        N(K_r, G)
        = |G[S]|\cdot N(K_{r-2}, G[\overline{S}]) + |S|\cdot N(K_{r-1}, G[\overline{S}]) + N(K_{r}, G[\overline{S}]). 
    \end{align*}
\end{fact}
\begin{fact}\label{FACT:formula-count-Kr-2-parts}
    Suppose that $G$ is a graph obtained from $K[V_1, \ldots, V_{\ell}]$ by adding a triangle-free graph. Let $S:= V_1 \cup V_2$ and $\overline{S} := V(G)\setminus S$. 
    Then 
    \begin{align*}
        N(K_{r}, G)
        \ =\ &  |G[V_1]|\cdot |G[V_2]|\cdot N(K_{r-4}, G[\overline{S}])  \\
          &  + \left(|G[V_1]|\cdot |V_2| + |G[V_2]|\cdot |V_1|\right) \cdot N(K_{r-3}, G[\overline{S}]) \\
          &  + |G[S]|\cdot N(K_{r-2}, G[\overline{S}]) 
          + |S|\cdot N(K_{r-1}, G[\overline{S}])  
             + N(K_{r}, G[\overline{S}]). 
    \end{align*}
\end{fact}
\hide{
\begin{fact}\label{FACT:formula-count-Kr-3-parts}
    Let $G$ be a graph, $S \subseteq V(G)$, and $\overline{S} := V(G)\setminus S$. 
    Suppose that the induced subgraph $G[S]$ is $3$-partite, and the induced bipartite subgraph $G[S, \overline{S}]$ is complete. 
    Then 
    \begin{align*}
        N(K_{r}, G)
        = N(K_3, G[S]) \cdot N(K_{r-3}, G[\overline{S}]) 
            & + |G[S]|\cdot N(K_{r-2}, G[\overline{S}]) \\
            & + |S|\cdot N(K_{r-1}, G[\overline{S}])  
             + N(K_{r}, G[\overline{S}]). 
    \end{align*}
\end{fact}
}
\begin{fact}\label{FACT:formula-count-Kr-3-parts}
    Let $G$ be a graph, $S \subseteq V(G)$, and $\overline{S} := V(G)\setminus S$. 
    Suppose that the induced subgraph $G[S]$ is $3$-partite, and the induced bipartite subgraph $G[S, \overline{S}]$ is complete. 
    Then 
    \begin{align*}
        N(K_{r}, G)
        = N(K_3, G[S]) \cdot N(K_{r-3}, G[\overline{S}]) 
            & + |G[S]|\cdot N(K_{r-2}, G[\overline{S}]) \\
            & + |S|\cdot N(K_{r-1}, G[\overline{S}])  
             + N(K_{r}, G[\overline{S}]). 
    \end{align*}
\end{fact}
%
%
%

We will also use the following results.

\begin{lemma}\label{LEMMA:one-partially-full-part}
    Let $r\ge 4$ and let $n,e\in\mathbb{N}$ satisfy $t_{r-1}(n)<e\le \binom{n}{2}$. 
    Suppose that $G \in \mathcal{H}^{\min}_{0}(n,e)$ is a graph with a vertex partition $V(G) = B_1\cup \ldots \cup B_{k-1}$ such that $G$ is the union of $K[B_1, \ldots, B_{k-1}]$ with a triangle-free graph. 
    Then $G$ contains at most one  part $B_i$ which is \emph{partially full}, meaning that $0 < |G[B_i]|< t_{2}(|B_i|)$. 
\end{lemma}
\begin{proof}
    Suppose to the contrary that $G$ contains two partially full parts $B_i$ and $B_j$ for some $1 \le i < j \le k-1$. 
    Let $x:= |G[B_i]|$, $\sigma:= |G[B_i]|+|G[B_j]|$ and $H:=G[V(G)\setminus (B_i\cup B_j)]$. 
    Observe from Fact~\ref{FACT:formula-count-Kr-2-parts} that there exist constants $C_2, C_3, C_4$ depending on $|B_i|$, $|B_j|$ and $H$ (but not on $x$) such that 
    \begin{align*}
        N(K_r, G)
        = N(K_{r-4},H)\cdot x(\sigma-x) + C_2 x + C_3 (\sigma-x) + C_4
        =: P(x). 
    \end{align*}
    Let $G_i$ be the graph obtained from $G$ by moving one edge from $G[B_j]$ to $G[B_i]$ and rearranging the latter graph to be still $K_3$-free, which is possible by Mantel's theorem.   Similarly, let $G_j$ be the graph obtained from $G$ by moving one edge from $G[B_i]$ to $G[B_j]$.
    Note that $N(K_r, G_i) = P(x+1)$ and $N(K_r, G_j) = P(x-1)$. 
    Since $e>t_{r-1}(n)$, we have
    \begin{equation}
        P(x+1) + P(x-1) - 2P(x)
        = - 2 N(K_{r-4},H) 
        < 0.  \label{eq:P}
    \end{equation}
    Thus $\min\left\{N(K_r, G_i), N(K_r, G_j)\right\} < N(K_r, G)$, contradicting the minimality of $G$. 
\end{proof}
The following simple inequality from~\cite{LPS20} will be useful. For completeness, we include its short proof here. 
\begin{lemma}[{\cite[Lemma~4.5]{LPS20}}]\label{LEMMA:d}
    For all integers $a\ge 1$, $k \ge 2$, and $n \ge ak$, we have 
    \begin{align}\label{eq:d}
        a(n-a)+t_{k-1}(n-a)> (a-1)(n-a+1)+t_{k-1}(n-a+1). 
    \end{align}
\end{lemma}
\begin{proof}
    Let $a_1 \ge \dots \ge a_{k-1}$ denote the part sizes of $T_{k-1}(n-a)$. If we increase its number of vertices by one, then the part sizes of the new Tur\'{a}n graph, up to reordering, can be obtained by increasing $a_{k-1}$ by one. 
    Thus the difference between the expressions in~\eqref{eq:d} is
    \begin{align}
        |K_{a_1, \ldots, a_{k-1},a}| - |K_{a_1, \ldots, a_{k-2},a_{k-1}+1, a-1}|
        = a_{k-1}a-(a_{k-1}+1)(a-1)
        = a_{k-1}-a+1, 
    \end{align}
    which is positive since $a_{k-1} \ge \floor*{(n-a)/(k-1)} \ge \floor*{(ak-a)/(k-1)}= a$.  
\end{proof}

\section{Extremal graphs in \texorpdfstring{$\mathcal{H}_0(n,e)$}{H0}}\label{se:H0}
%
As an intermediate step towards Theorem~\ref{THM:N(Kr,G)-H1}, we will first prove the following result, which determines the extremal graphs in $\mathcal{H}_{0}(n,e)$. 

\begin{proposition}\label{PROP:H0-min-H1-ast}
        For all integers $n \ge r\ge 4$ and $t_{r-1}(n)<e\le \binom{n}{2}$, we have that $\mathcal{H}^{\min}_{0}(n,e) = \mathcal{H}^{\ast}_{1}(n,e)$. 
\end{proposition}

We will use this result later to prove Theorem~\ref{THM:N(Kr,G)-H1} by induction on the number of parts in a graph in $\mathcal{K}(n,e)$.
Note that, in general, neither $\mathcal{K}(n,e)$ nor $\mathcal{H}_0(n,e)$ is a subfamily of the other. However, when we work on the structure of extremal graphs in $\mathcal{K}(n,e)$ in the proof of Theorem~\ref{THM:N(Kr,G)-H1}, some intermediate graphs may be in~$\mathcal{H}_0(n,e)$.
 
We need some further preliminaries before we can prove Proposition~\ref{PROP:H0-min-H1-ast}.

Given a graph $G \in \mathcal{H}^{\min}_{0}(n,e)$ with partition $B_1, \ldots, B_{k-1}$, we apply the following modification to $G$ to obtain a new graph $H'= H'(G) \in \mathcal{H}^{\min}_{0}(n,e)$. Note that, in fact, these steps do not depend on~$r$.
\begin{description}
    \item[\textbf{Step 1:}] If there is a part $B_i$ that is partially full in $G$, then let $B:= B_i$ (by Lemma~\ref{LEMMA:one-partially-full-part}, such $B_i$ is unique if it exists). 
    Otherwise, take an arbitrary $i\in [k-1]$ with $|G[B_i]| = t_2(|B_i|)$ and let $B:= B_i$. 
    Since $|G| > t_{k-1}(n)$, $|G[B_i]|$ cannot be $0$ for all $i \in [k-1]$. Thus, the set $B$ is well-defined. 
    \item[\textbf{Step 2:}] Note that $G-B$ is a complete multipartite graph. 
    Let $A_1, \ldots, A_{t-2}$ denote its parts. 
    Let $a_i := |A_i|$ for $i\in [t-2]$ and assume that $a_1 \ge \dots \ge a_{t-2}$. 
    Note that each original part $B_\ell$ is either $B$, some $A_i$, or the union of two parts $A_i$ and $A_j$.
    \item[\textbf{Step 3:}] Choose integers $a_{t-1} \ge a_t \ge 1$ such that 
    \begin{align*}
        a_{t-1}+a_{t} = |B| \quad\text{and}\quad 
        (a_{t-1}+1)(a_t-1) < |G[B]| \le a_{t-1}a_{t}. 
    \end{align*}
    Note that this is possible by Mantel's theorem since $G[B]$ is triangle-free. 
    Let $A_{t-1}\sqcup A_t = B$ be a partition with $|A_{t-1}|=a_{t-1}$ and $|A_{t}| = a_t$. 
    If $|G[B]| = t_2(|B|)$, then $a_{t-1}=\lceil |B|/2\rceil$ and $a_t=\lfloor |B|/2\rfloor$ and we assume that $A_{t-1}\sqcup A_t = B$ is the original partition of $G[B]$ with the two parts labelled so that $|A_{t-1}|\ge |A_t|$.
    \item[\textbf{Step 4:}] Let $H'$ be obtained from $K[A_1, \ldots, A_t]$ by removing a star whose centre lies in $A_t$ and $m'$ leaves lie in $A_{t-1}$, 
    where
    \begin{equation}\label{eq:m'}
     m':= \sum_{ij\in \binom{[t]}{2}}a_ia_j - e = a_{t-1}a_t - |G[B]|.
     \end{equation} 
    This is possible because, by Step 3, 
    \begin{align}
    \label{eq:m'Ineq}
        0 \le m' 
        = a_{t-1}a_t - |G[B]|
        \le a_{t-1}a_{t} - \left((a_{t-1}+1)(a_t-1)+1\right)
        = a_{t-1}-a_t. 
    \end{align}
    \end{description}
Notice that to obtain $H'$ we only change the structure of $G$ on $B$ while keeping $|G[B]| = |H'[B]|$. Thus $H'\in \mathcal{H}_0(n,e)$ and, since $G[B, V(G)\setminus B]$ is complete bipartite and $G[B]$ is triangle-free, it follows from Fact~\ref{FACT:formula-count-Kr-1-part} that $N(K_r, H') = N(K_r, G)$, and hence, $H'\in \mathcal{H}_0^{\min}(n,e)$. 
\begin{lemma}\label{LEMMA:e}
    For all $r\ge 3$, integers $n$ and $e$ with $t_{r-1}(n)<e\le \binom{n}{2}$ and $G\in \mathcal{H}_{0}^{\min}(n,e)$, the graph $H'$ produced by Steps 1--4 above is isomorphic to $H^{\ast}(n,e)$. 
\end{lemma}
\begin{proof}
    To prove that $H' \cong H^{\ast}(n,e)$, it suffices to show that $t = k$ and $(|A_1|, \ldots, |A_{t}|) = \bm{a}^{\ast}$, where $k$ and $\bm{a}^{\ast}$ are as defined in Section~\ref{SEC:introduction}.  

    \begin{claim}\label{CLAIM:Lemma:a-1}
        If $m'=0$, then $|H'[A_{h}\cup A_{i} \cup A_j]| > t_{2}(a_h+a_i+a_j)$ for all $\{h,i,j\}\in \binom{[t]}{3}$. 
        If $m'>0$, then $|H'[A_{h}\cup A_{t-1} \cup A_t]| > t_{2}(a_h+a_{t-1}+a_t)$ for all $h\in [t-2]$. 
    \end{claim}
    \begin{proof}
        Let $S:= A_{h}\cup A_{i} \cup A_j$, with $\{i,j\}=\{t-1,t\}$ if $m'>0$. 
        Suppose to the contrary that $|H'[S]| \le t_{2}(|S|)$. 
        Then let $G_1$ be a new graph obtained from $H'$ by replacing $H'[S]$ with a bipartite graph of the same size. 
        Note that the induced bipartite graph $H'[S, \overline{S}]$ is complete. (Indeed, this is trivially true if $m'=0$ as then $H' = K[A_{1}, \ldots, A_{t}]$; if $m'> 0$, then the only non-complete pair is $[A_{t-1}, A_{t}]$, but both sets lie in $S$.)
        Since $H'$ is $t$-partite, the graph $G_1$ is $(t-1)$-partite (and with at most one non-complete pair of parts). 
        By Steps 2--3, we have $t \le 2(k-1)$. 
        So we can represent $G_1$ as the union of a complete $(k-1)$-partite graph and a triangle-free graph, which implies that $G_1 \in \mathcal{H}_0(n,e)$. 
        It is easy to see from Fact~\ref{FACT:formula-count-Kr-3-parts} that $N(K_r, G_1) \le N(K_r, H')$, since $0 = N(K_3, G_1[S]) \le N(K_3, H'[S])$. 
        So it follows from the minimality of $H'$ that $N(K_3, H'[S]) = 0$. 
        If $\{t-1, t\}$ is not a subset of $\{h,i,j\}$, then $H'[S]$ is a complete $3$-partite graph and contains at least one traingle, contradicting $N(K_3, H'[S]) = 0$. 
        Therefore, $\{t-1, t\} \subseteq \{h,i,j\}$. 
        By symmetry, we may assume that $\{t-1, t\} = \{i,j\}$ (thus being consistent with our earlier assumption if $m'>0$). 
        Note that $|H[A_{t-1}, A_{t}]| \ge 1$, since otherwise, we would have $m' \ge a_{t-1}a_{t} > a_{t-1}-a_{t}$, contradicting~\eqref{eq:m'Ineq}. 
        Note that each edge in $H[A_{t-1}, A_{t}]$ is in $|A_{h}|$  triangles in $H[S]$, contradicting $N(K_3, H'[S]) = 0$. 
    \end{proof}
    \begin{claim}\label{CLAIM:Lemma:a-2}
        If $m'>0$, then $a_{t-2} \ge a_{t-1}$. 
    \end{claim}
    \begin{proof}
        Suppose to the contrary that $a_{t-2} \le a_{t-1}-1$. 
        Then let $G_2$ be a new graph obtained from $H'$ by moving edges from $[A_{t-2}, A_{t}]$ to $[A_{t-1}, A_{t}]$ until this is no longer possible. 
        Let $S:= A_{t-2} \cup A_{t-1} \cup A_{t}$. 
        If $A_{t-2} \cup A_t$ is an independent set in $G_2$ (i.e.\ if $m' \ge a_{t-2}a_{t}$),  then $|H'[S]| = |G_2[S]| \le t_2(|S|)$, contradicting Claim~\ref{CLAIM:Lemma:a-1}. 
        Thus $G_2[S]$ can be viewed as a graph obtained from $K[A_{t-2}, A_{t-1}, A_{t}]$ by removing $m'$ edges from $K[A_{t-2}, A_{t}]$. 
        So $G_2 \in \mathcal{H}_0(n,e)$. 
        Note that 
        \begin{align*}
            N(K_3, G_2[S]) - N(K_3, H'[S])
            = m'\left(a_{t-2} - a_{t-1}\right)
            <0, 
        \end{align*}
        which combined with Fact~\ref{FACT:formula-count-Kr-3-parts} implies that $N(K_r, G_2) - N(K_r, H')<0$, contradicting the minimality of $H'$. 
    \end{proof}
    If $m'>0$, let $C_i:= A_i$ for $i\in [t]$. 
    If $m'=0$, let $C_1, \ldots, C_t$ be a relabelling of $A_1, \ldots, A_t$ so that the sizes of the sets are non-increasing. 
    Regardless of the value of $m'$, the following statements clearly hold: 
    \begin{enumerate}[label=(\roman*)]
        \item $c_1 \ge \dots \ge c_{t}$, where $c_i := |C_i|$ for $i\in [t]$, 
        \item $0 \le m' \le c_{t-1}-c_{t}$,
        \item Claim~\ref{CLAIM:Lemma:a-1} applies to all triples $\{C_i, C_{t-1}, C_{t}\}$ for $i\in [t-2]$. 
    \end{enumerate}
    The rest of the proof is written so that it works for both $m'=0$ and $m'>0$. 
    \begin{claim}\label{CLAIM:Lemma:a-3}
        We have $c_1 \le c_{t-1} + 1$. 
    \end{claim}
    \begin{proof}
        Let $S:= C_1 \cup C_{t-1} \cup C_{t}$. 
        Note that 
        \begin{align*}
            |K_{c_1-1, c_{t-1}+1, c_{t}}| - |H'[S]|
            = m' - c_{t-1}+c_1-1 
            =: m''. 
        \end{align*}
        Suppose to the contrary that $c_1 \ge c_{t-1}+2$. 
        Then $m'' \ge m'+1$.  
        Take a partition $C'_{1} \cup  C'_{t-1} \cup C'_{t} = S$ of sizes $c_1-1, c_{t-1}+1, c_{t}$, respectively. 
        Let $H_{S}$ be the graph obtained from $K[C'_{1}, C'_{t-1}, C'_{t}]$ by removing $m''$ edges between $C'_{t-1}$ and $C'_{t}$. This is possible since $m'' \le (c_{t-1}+1)c_{t}$. (Indeed, otherwise $|H'[S]| \le (c_1-1)(c_{t-1}+c_{t}+1) \le t_{2}(|S|)$, contradicting Claim~\ref{CLAIM:Lemma:a-1}.) We have $|H_{S}| = |H'[S]|$. 
        Let $H''$ be the graph obtained from $H'$ by replacing $H'[S]$ with $H_{S}$. Note that $H'' \in \mathcal{H}_{0}(n,e)$. 
        It follows from $m' \le c_{t-1}-c_{t}$ that  
        \begin{eqnarray*}
            N(K_3, H'[S]) - N(K_3, H''[S])
            & =& \left(c_1c_{t-1}c_{t} - m'c_1\right) \\
            &  -& \left((c_1-1)(c_{t-1}+1)c_{t} - (m'-c_{t-1}+c_1-1)(c_1-1)\right) \\
            & \ge& (c_1 - c_{t})(c_1-c_{t-1}-2)+1
            \ \ge  1, 
        \end{eqnarray*}
        which combined with Fact~\ref{FACT:formula-count-Kr-3-parts} implies that $N(K_r, H') - N(K_r, H'') > 0$, contradicting the minimality of $H'$. 
    \end{proof}
    \begin{claim}\label{CLAIM:Lemma:a-4}
        We have $t = k$.  
    \end{claim}
    \begin{proof}
        It suffices to show that $t_{t-1}(n)< e \le t_{t}(n)$. 
        The upper bound $e \le t_{t}(n)$ is trivial, since $H'$ is $t$-partite. 
        So it remains to show that $e > t_{t-1}(n)$. 
        Let $T:= H'[C_1 \cup \dots \cup C_{t-1}]$. 
        It follows from Claim~\ref{CLAIM:Lemma:a-3} that  $T \cong T_{t-1}(n-c_t)$. 
        Therefore, 
        \begin{align}\label{equ:size-H'}
            |H'| - t_{t-1}(n-c_t) 
            = |H'\setminus T|
            = c_{t}(n-c_t) - m'. 
        \end{align}
        On the other hand, by viewing $T_{t-1}(n)$ as a graph obtained from $T_{t-1}(n-c_t)$ by adding $c_t$ new vertices into some parts, we obtain  
        \begin{align*}
            t_{t-1}(n) - t_{t-1}(n-c_t) 
            \le c_{t}(n-c_{t-1}-1). 
        \end{align*}
        By combining these two inequalities, we obtain 
        \begin{align*}
            |H'|-t_{t-1}(n)
            \ge c_{t}(c_{t-1}+1-c_{t}) - m'
            \ge (c_t-1)(c_{t-1}-c_t)+c_t 
            > 0, 
        \end{align*}
        proving that $e > t_{t-1}(n)$. 
    \end{proof}
    \begin{claim}\label{CLAIM:Lemma:a-5}
        The sequence $\left(|C_1|, \ldots, |C_{k}|\right)$ of part sizes is equal to $\bm{a}^{\ast} = \bm{a}^{\ast}(n,e)$. 
    \end{claim}
    \begin{proof}
        Recall that $t=k$ and, by~\eqref{equ:size-H'}, we have that 
        \begin{align}
            |H'| - t_{k-1}(n-c_k)
            = c_k(n-c_k) - m'
            \le c_k (n-c_k). \label{eq:a=ck}
        \end{align}
        
Let us show that $c_k$ is the smallest nonnegative integer $a$ satisfying 
$$
 f(a):= a(n-a) + t_{k-1}(n-a) \ge e.
 $$
 This inequality holds for $a=c_k$ by~\eqref{eq:a=ck}.
Note that $c_k \le n/k$ as it is the smallest among $c_1 + \dots + c_k = n$. 
        Thus, by Lemma~\ref{LEMMA:d}, it is enough to check that $a=c_k-1$ violates this condition. 
        Notice that 
        \begin{align*}
            f(c_k-1) - f(c_k)
            \le 2c_{k}-n-1 + (n-c_k-c_{k-1})
            = c_k - c_{k-1} - 1. 
        \end{align*}    
        Therefore, it follows from  $m' \le c_{k-1}-c_{k}$ that 
        \begin{align*}
            f(c_k-1) 
            \le f(c_k) - (m'+1) 
            \le |H'| + m' -(m'+1)
            < |H'|, 
        \end{align*}
        as desired. 
        
        Thus $c_k=a_k^{\ast}$ and (since $t=k$ by Claim~\ref{CLAIM:Lemma:a-4}) we have $(c_1,\dots,c_k)=\bm{a}^{\ast}$ by Claim~\ref{CLAIM:Lemma:a-3}, as desired.
    \end{proof}
    %
    Also, it follows from the definitions that $m' = m^{\ast}$ and thus $H'$ is isomorphic to $H^{\ast}(n,e)$.  
    This completes the proof of Lemma~\ref{LEMMA:e}. 
\end{proof}
    Now we are ready to prove Proposition~\ref{PROP:H0-min-H1-ast}. 
    %
\begin{proof}[Proof of Proposition~\ref{PROP:H0-min-H1-ast}]
    Let $G\in \mathcal{H}^{\min}_{0}(n,e)$ be arbitrary. 
    Let $B_1, \ldots, B_{k-1}$ be a vertex partition such that $G$ is the union of $K[B_1, \ldots, B_{k-1}]$ with a triangle-free graph $J$. 
    Let $b_i:= |B_i|$ for $i\in [k-1]$. 
    Apply Steps 1--4 to $G$ to obtain a $k$-partite graph $H'$ with parts $A_1, \ldots, A_{k}$. By Lemma~\ref{LEMMA:e}, we have $H' \cong H^{\ast} := H^{\ast}(n,e)$.
    Assume that $|A_i| = a^{\ast}_i$ for $i\in [k]$ and that all missing edges of $H'$ (if any exist) go between $A_{k-1}$ and $A_{k}$. 

    The following claim follows from the definitions of Steps 1--4. 
    \begin{claim}\label{CLAIM:LEMMA:H0-min-H1-ast-1}
        If $i\in [k-1]$ satisfies $|G[B_i]| \in \left\{0, t_{2}(b_i)\right\}$, then $H'[B_i] = G[B_i]$. 
    \end{claim}

    Since $H'$ is $k$-partite, it follows from the definitions of Steps 1--4 that exactly one part $B_p$ of $G$ is divided into $A_q \cup A_s$ in Steps 2--3, where, say, $1\le q < s \le k$, while the remaining parts of $G$ correspond to the remaining parts of~$H'$. 
    In particular, $b_p = a_{q}^{\ast} + a_{s}^{\ast}$. 

    \begin{claim}\label{CLAIM:LEMMA:H0-min-H1-ast-2}
        We have $|G[B_p]|>0$. 
    \end{claim}
    \begin{proof}
        It follows from $m^{\ast} \le a_{k-1}^{\ast} - a_{k}^{\ast}$ that 
        \begin{align*}
            |H'[B_p]| 
            = a_{q}^{\ast} a_{s}^{\ast} - m^{\ast}
            \ge a_{q}^{\ast} a_{s}^{\ast} - (a_{k-1}^{\ast} - a_{k}^{\ast})
            > 0. 
        \end{align*}
        Combined with Claim~\ref{CLAIM:LEMMA:H0-min-H1-ast-1}, we see that $|G[B_p]|>0$. 
    \end{proof}

    Suppose first that $m^{\ast} = 0$. Then $H' = K[A_1, \ldots, A_{k}]$, and $G$ can be obtained from $H'$ by replacing $H'[A_{q}\cup A_{s}]$ with $G[B_p]$. Moreover, $G[B_{p}]$ is a triangle-free graph with $a^{\ast}_{q} + a^{\ast}_{s}$ vertices and $a^{\ast}_{q} a^{\ast}_{s}$ edges. 
    If $a^{\ast}_{s} = a^{\ast}_{k}$, then it follows from the definition of $\mathcal{H}_{1}^{\ast}(n,e)$ that $G \in \mathcal{H}_{1}^{\ast}(n,e)$. 
    Otherwise, $|a^{\ast}_{q} - a^{\ast}_{s}| \le 1$ (by the definition of $\bm{a}^{\ast}$), and hence, $G[B_p]\cong T_{2}(a^{\ast}_{q} + a^{\ast}_{s})$. 
    This implies that $G$ does not contain any partially full part, and hence, $G = H' \in \mathcal{H}_{1}^{\ast}(n,e)$. 

    Suppose that $m^{\ast} > 0$. 
    Since $G[A_i,A_j]$ is complete for all $\{i,j\} \neq \{q,s\}$ and $H'[A_i,A_j]$ is complete iff $\{i,j\} \neq \{k-1, k\}$, we have $\{q,s\} = \{k-1, k\}$. 
    Thus $G$ can be obtained from $K[A_1, \ldots, A_{k}]$ by replacing $K[A_{k-1} \cup A_k]$ with a triangle-free graph with $a^{\ast}_{k-1}a^{\ast}_k - m^{\ast}$ edges. 
    This gives $G \in \mathcal{H}_{1}^{\ast}(n,e)$.
    We conclude that $\mathcal{H}^{\min}_{0}(n,e) \subseteq \mathcal{H}^{\ast}_{1}(n,e)$. 
    Since $\mathcal{H}^{\ast}_{1}(n,e) \subseteq \mathcal{H}_0(n,e)$ and every graph in $\mathcal{H}^{\ast}_{1}(n,e)$ contains the same number of $K_r$'s, we have $\mathcal{H}^{\min}_{0}(n,e) = \mathcal{H}^{\ast}_{1}(n,e)$. 
    \end{proof}

\section{Proof of Theorem~\ref{THM:N(Kr,G)-H1}}\label{SUBSEC:proof-Thm}
%
%
With Proposition~\ref{PROP:H0-min-H1-ast} in hand, we are now ready to prove Theorem~\ref{THM:N(Kr,G)-H1}. 

\begin{proof}[Proof of Theorem~\ref{THM:N(Kr,G)-H1}]
    Fix integers $n \ge r \ge 3$ and $e \le \binom{n}{2}$. 
    Notice that~\eqref{eq:N(Kr,G)-H1} can be reduced to $\min\big\{N(K_r, G) \colon G\in \mathcal{K}(n,e)\big\} \ge  h_{r}^{\ast}(n,e)$, since the other direction is trivially true. 
    Suppose that $G \in \mathcal{K}^{\min}(n,e)$ is a graph obtained from a complete $\ell$-partite graph by adding a triangle-free graph to one part. 
    We aim to show that $N(K_r, G) \ge h_{r}^{\ast}(n,e)$ when $r\ge 3$ and, in addition, $G \in \mathcal{H}_{1}^{\ast}(n,e)$ when $r \ge 4$ and $e> t_{r-1}(n)$. 
    We prove this statement by induction on $\ell + r$.
    Notice that if $\ell = k-1$ (where $k= k(n,e)$) and $r \ge 4$, then $G \in \mathcal{H}_{0}(n,e)$, and it follows from Proposition~\ref{PROP:H0-min-H1-ast} that $G \in \mathcal{H}_{1}^{\ast}(n,e)$, as desired. 
    If $\ell= k-1$ and $r=3$, then $G \in \mathcal{H}_{0}(n,e)$, and it follows from~{\cite[Proposition~1.5]{LPS20}} that $N(K_3, G) \ge h_{3}^{\ast}(n,e)$.
    So the statement is true for all pairs $(\ell, r)$ with $\ell = k-1$ and $r\ge 3$, and this serves as our base case.  
    
    Assume that $\ell \ge k$ and $r \ge 3$. 
    Let $U_1 \cup \dots \cup U_{\ell} = V(G)$ be a partition such that $G$ is obtained from the complete $\ell$-partite graph $K[U_1, \ldots, U_{\ell}]$ by adding a triangle-free graph into $U_{\ell}$. We can assume that $U_{\ell}$ is not an independent set (otherwise consider instead the $(\ell-1)$-partition of $V(G)$ where $U_{\ell-1}$ and $U_\ell$ are merged together). 
    
First, we prove~\eqref{eq:N(Kr,G)-H1}. Assume that $\ell\ge r-1$, as otherwise  $h_r^{\ast}(n,e)=0$ and there is nothing to do.    
    Note that
       $U_{\ell}$ is as large as any other part: if some part $U_i$ has strictly larger size then by moving all edges from $U_\ell$ to $U_i$ (by $|U_i|>|U_\ell|$ there is enough space for this) we strictly decrease the number of $r$-cliques (since $\ell \ge r-1$), a contradiction.
By relabelling parts $U_1, \ldots, U_{\ell-1}$, we may assume that $U_{1}$ is of smallest size among $U_1, \ldots, U_{\ell-1}$. 
    Let $\hat{G}$ denote the induced subgraph of $G$ on $U_2 \cup \dots \cup U_{\ell}$. 
    Let $\hat{n} := n-|U_{1}|$ and $\hat{e} := |\hat{G}|$. 
    Let $\hat{k} := k(\hat{n}, \hat{e})$ be as defined in~\eqref{equ:def-k} (while we reserve $k$ for $k(n,e)$). 
    
    \begin{claim}\label{CLAIM:hat-k<k}
        We have $\hat{k} \le k$. 
    \end{claim}
    \begin{proof}
        Let $H^{\ast} = H^{\ast}(n,e)$ be the $k$-partite graph as defined in Section~\ref{SEC:introduction}. 
        Assume that $A_1^{\ast}, \ldots, A_{k}^{\ast}$ are the corresponding parts of $H^{\ast}$ of sizes $a_1^{\ast} \ge \dots \ge a_{k}^{\ast}$, respectively. 
        It is clear that $|A_1^{\ast}| \ge \frac{n}{k}$. 
        It follows from the minimality of $U_{1}$ that $|U_{1}| \le \frac{n-|U_{\ell}|}{\ell-1} \le \frac{n}{k} \le |A_1^{\ast}|$. 
        Let $W_1 \subseteq A_1^{\ast}$ be a set of size $|U_1|$ and let $H'$ be the induced subgraph of $H^{\ast}$ on $V(H)\setminus W_1$. 
        Observe that $H'$ is still a $k$-partite graph and $|H'| \ge |\hat{G}|$.  
        So it follows from the definition that $\hat{k} \le k$. 
    \end{proof}

    Note that $\hat{G}$ can be viewed as a graph obtained from a complete $(\ell-1)$-partite graph by adding a triangle-free graph into one part; in particular, $\hat{G}\in\mathcal{K}(\hat{n},\hat{e})$.
    Let $\hat{H}$ be $H^{\ast}(\hat{n}, \hat{e})$ and let $G'$ be the graph obtained from $G$ by replacing $\hat{G}$ with $\hat{H}$. 
    It follows from the inductive hypothesis that 
    \begin{align*}
        N(K_r, \hat{H}) = h_{r}^{\ast}(\hat{n},\hat{e}) \le  N(K_r, \hat{G})
        \quad\text{and}\quad 
        N(K_{r-1}, \hat{H}) \le N(K_{r-1}, \hat{G}).
    \end{align*}
    Hence,
    \begin{align*}
        h_{r}^{\ast}(n,e)
        \le N(K_r, G')
        & = N(K_r, \hat{H}) + |U_1| \cdot N(K_{r-1}, \hat{H})  \\
        & \le N(K_r, \hat{G}) + |U_1| \cdot N(K_{r-1}, \hat{G})
        = N(K_r, G), 
    \end{align*}
    finishing the inductive step for proving~\eqref{eq:N(Kr,G)-H1}. 

    Now suppose that that $r \ge 4$ and $e> t_{r-1}(n)$, and suppose for contradiction that $G\not\in \mathcal{H}_1^{\ast}(n,e)$. Reusing the notation introduced above, 
    let us first derive a contradiction from assuming that
  $\hat{G} \not\in \mathcal{H}_1^{\ast}(\hat{n}, \hat{e})$.

If $\hat{e} > t_{r-1}(\hat{n})$, then it follows from the inductive hypothesis that  
    \begin{align*}
        N(K_r, \hat{H}) <  N(K_r, \hat{G})
        \quad\text{and}\quad 
        N(K_{r-1}, \hat{H}) \le N(K_{r-1}, \hat{G}). 
    \end{align*}
    Therefore, 
    \begin{align}
        N(K_r, G')
        & = N(K_r, \hat{H}) + |U_1| \cdot N(K_{r-1}, \hat{H})  \notag\\
        & < N(K_r, \hat{G}) + |U_1| \cdot N(K_{r-1}, \hat{G})
        = N(K_r, G),\label{eq:NKrG'}
    \end{align}
    contradicting the minimality of $G$. 
    
So suppose that $\hat{e} \le t_{r-1}(\hat{n})$. We have that $\ell \ge k \ge r$. Recall that $\hat{G}$ is a graph obtained from an $(\ell-1)$-partite graph by adding a non-empty triangle-free graph. Thus, we have $N(K_r, \hat{H}) = 0 <  N(K_r, \hat{G})$.
    In addition, by~\eqref{eq:N(Kr,G)-H1}, we have $N(K_{r-1}, \hat{H}) = h_{r-1}^{\ast}(\hat{n}, \hat{e}) \le N(K_{r-1}, \hat{G})$. But then the same calculation as in~\eqref{eq:NKrG'} gives a contradiction to the minimality of~$G$.
    
Thus we have that $\hat{G} \in \mathcal{H}_1^{\ast}(\hat{n}, \hat{e})$.
    Let $\hat{A}_1^{\ast} \cup \ldots \cup \hat{A}_{\hat{k}}^{\ast} = V(\hat{G})$ be the partition of $\hat{G}$ as in the definition of~$\mathcal{H}_{1}^{\ast}(\hat{n},\hat{e})$. 
    Let $B_1 := U_1 \cup \hat{A}_1^{\ast}$, $B_i := \hat{A}_i^{\ast}$ for $2\le i\le \hat{k}-2$, and $B_{\hat{k}-1}:=\hat{A}_{\hat{k}-1}\cup \hat{A}_{\hat{k}}$. 
    We can view $G$ as a graph obtained from $K[B_1, \ldots, B_{\hat{k}-1}]$ by adding triangle-free graphs into two parts, namely $G[B_1]$ and $G[B_{\hat{k}-1}]$.
    Since $\hat{k} \le k$ by Claim~\ref{CLAIM:hat-k<k}, it holds
    that $G \in \mathcal{H}_{0}(n,e)$. 
    Therefore, it follows from Proposition~\ref{PROP:H0-min-H1-ast} that $G \in \mathcal{H}_{1}^{\ast}(n,e)$, finishing the proof of Theorem~\ref{THM:N(Kr,G)-H1}.
\end{proof}

Let us remark that if we replace the family $\mathcal{K}(n,e)$ in Theorem~\ref{THM:N(Kr,G)-H1} by the larger family $\mathcal{K}'(n,e)$ that  consists of all graphs obtained from a complete partite graph by adding a triangle-free graph (that is, we allow to add edges into more than one part)  then the theorem will remain true. Indeed,  for $r\ge 4$, the proof of  Lemma~\ref{LEMMA:one-partially-full-part} (which in fact works for any number of parts) shows that every extremal graph  $\mathcal{K}'(n,e)$ has at most one partially full part and thus belongs to~$\mathcal{K}(n,e)$. For $r=3$, 
the equality in~\eqref{eq:N(Kr,G)-H1}, will also remain true (again by  the proof of  Lemma~\ref{LEMMA:one-partially-full-part} except the inequality in~\eqref{eq:P} becomes equality).

\section{Proof of Proposition~\ref{PROP:N(K_r,H)-H2-ast}}\label{SUBSEC:proof-Prop}
\begin{proof}[Proof of Proposition~\ref{PROP:N(K_r,H)-H2-ast}]
    First, we prove that $N(K_r, H) = h_{r}^{\ast}(n,e)$ for all $H\in \mathcal{H}_{2}^{\ast}(n,e)$. 
    Fix $H \in \mathcal{H}_{2}^{\ast}(n,e)$. 
    
    First consider the case when $(|A_1|,\dots,|A_k|)=\bm{a}^{\ast}$, where the sets $A_1,\dots,A_k$ are as in the definition of $\mathcal{H}_{2}^{\ast}(n,e)$. 
    Let $K := K[A_1, \ldots, A_k]$, and $m_i^{\ast} := |\overline{H}[B_i, A_i]|$ for $i\in I:= \left\{j \in [k-1] \colon |A_j| = |A_{k-1}|\right\}$. 
    Note from the definition of $I$ that for all $i\in I$, we have that
    \begin{align*}
        N(K_{r-2}, K[A_1, \ldots, A_{i-1}, A_{i+1}, \dots, A_{k-1}]) = N(K_{r-2}, K[A_1, \ldots, A_{k-2}]), 
    \end{align*}
     because we count $r$-cliques in two isomorphic graphs.
    Therefore, 
    \begin{align}\label{equ:PROP-K-H}
        N(K_r, K) - N(K_r, H)
        & = \sum_{i\in I}m_i^{\ast} \cdot N(K_{r-2}, K[A_1, \ldots, A_{i-1}, A_{i+1}, \dots, A_{k-1}]) \notag \\
        & = \sum_{i\in I}m_i^{\ast} \cdot N(K_{r-2}, K[A_1, \ldots, A_{k-2}]) \notag \\
        & = m^{\ast} \cdot N(K_{r-2}, K[A_1, \ldots, A_{k-2}])
        = N(K_r, K) - N(K_r, H^{\ast}).   
    \end{align}
    It follows that $N(K_r, H) = N(K_r, H^{\ast}) = h^{\ast}(n,e)$, as desired.

    Now suppose that $(|A_1|,\dots,|A_k|)\not=\bm{a}^{\ast}$. Recall that then $m^{\ast} = 0$, $(|A_{1}|, \ldots, |A_{k}|) = (a_{2}^{\ast}, \ldots, a_{k-1}^{\ast}, a_{1}^{\ast}-1, a_{k}^{\ast}+1)$, $m = a_{1}^{\ast} - a_{k}^{\ast} +1$, and $H$ is a graph obtained from $K[A_1, \ldots, A_{k}]$ by removing some $m$ edges.  
    We may assume that these $m$ edges were removed from parts $[A_{k-1}, A_{k}]$, since this does not affect the value of $N(K_r, H)$ by the calculation in~\eqref{equ:PROP-K-H}. 
    Now, by viewing $H$ as a graph obtained from $K[A_{1}, \ldots, A_{k}]$ by replacing $K[A_{k-1}, A_k]$ with a triangle-free graph, we see that $H \in \mathcal{H}_{1}^{\ast}(n,e)$, and hence, $N(K_r, H) = h^{\ast}(n,e)$.

    Next, we show that there are infinitely many pairs $(n,e) \in \mathbb{N}^2$ with $t_{r-1}(n) < e \le \binom{n}{2}$ such that $\mathcal{H}_{2}^{\ast}(n,e) \setminus \mathcal{H}_{1}^{\ast}(n,e) \neq \emptyset$. It is enough to chose $(n,e)$ so that $a_{k-2}^{\ast}=a_{k-1}^{\ast}$ and $m^{\ast},a_k^{\ast}\ge 2$; the choice that we use (in~\eqref{eq:OurChoice} below) is rather arbitrary.

Take any integers $p \ge r-1$, $q \ge 100$, and $2 \le m \le q$. 
    Let $n := 2pq+ q$ and $e := \binom{p}{2}(2q)^2 + 2pq^2 - m$. Note that $e+m$ is the number of edges in the complete $(p+1)$-partite graph $K_{2q,\dots,2q,q}$ with $p$ parts of size $2q$ and one part of size~$q$.
    The choice of $(p,q,m)$ ensures that 
    \begin{align*}
        e=\binom{p}{2}(2q)^2 + 2pq^2 - m
        > \binom{p}{2}\left(\frac{2pq+q}{p}\right)^2 
        \ge t_{p}(n). 
    \end{align*}
    By $e< e+m\le t_{p+1}(n)$, we have that $k(n,e) = p$. 
    
Let us show that $a_{p}^{\ast} = q$. By Lemma~\ref{LEMMA:d}, it is enough to show that $(q-1)(n-q-1)+t_{k-1}(n-q-1)<e$. The left-hand side here is the size of the graph obtained from the complete partite graph $K_{2q,\dots,2q,q}$ by moving a vertex from the part of size $q$ into one of size $2q$. This results in losing $q+1>m$ edges, giving the required. Thus
        \begin{align}
        a_1^{\ast} = \dots = a^{\ast}_{p-1} = 2q, \quad 
        a_{p}^{\ast} = q, \quad\text{and}\quad  
        m^{\ast} = m.\label{eq:OurChoice}
    \end{align}
    Let $V_1 \cup \dots \cup V_{p+1} = [n]$ be a partition such that $|V_1| = \dots = |V_p| = 2q$ and $|V_{p+1}|=q$. 
    Fix $m$ distinct vertices $v_1, \ldots, v_m \in V_{p+1}$, and choose a vertex $u_i \in V_i$ for every $i\in [m]$. 
    Let $G$ be the graph obtained from $K[V_1, \ldots, V_{p-1}]$ by removing pairs in $\{\{v_i, u_i\} \colon i\in [m]\}$. 
    It is easy to see that $G \in \mathcal{H}_{2}^{\ast}(n,e) \setminus \mathcal{H}_{1}^{\ast}(n,e)$, proving Proposition~\ref{PROP:N(K_r,H)-H2-ast}. 
\end{proof}

\section*{Acknowledgement}
We would like to thank the anonymous referee for helpful comments. 
\bibliographystyle{abbrv}
\bibliography{MinimumClique}
\end{document}